\documentclass{birkjour}
\theoremstyle{definition}
\newtheorem*{definition}{Definition}
\theoremstyle{remark}
\newtheorem{remark}{Remark}
\theoremstyle{theorem}
\newtheorem{theorem}{Theorem}
\theoremstyle{corollary}
\newtheorem{corollary}{Corollary}
\theoremstyle{proposition}
\newtheorem{proposition}{Proposition}
\theoremstyle{lemma}
\newtheorem{lemma}{Lemma}

\numberwithin{proposition}{section}
\numberwithin{lemma}{section}
\numberwithin{remark}{section}
\numberwithin{equation}{section}

\begin{document}

%
%
%
%
%
%
%
%
%

\title[Uniqueness of symmetric Navier-Stokes flows]
 {On uniqueness of symmetric Navier-Stokes flows around a body in the plane}

\author[T. Nakatsuka]{Tomoyuki Nakatsuka}

\address{%
Graduate School of Mathematics\\
Nagoya University\\
Nagoya 464-8602\\
Japan}

\email{m09033b@math.nagoya-u.ac.jp}
\subjclass{35Q30, 76D05}

\keywords{Stationary Navier-Stokes equation, Plane exterior domain, Uniqueness, Symmetry}

\date{}
\begin{abstract}
We investigate the uniqueness of symmetric weak solutions to the stationary Navier-Stokes equation in a two-dimensional exterior domain $\Omega$. It is known that, under suitable symmetry condition on the domain and the data, the problem admits at least one symmetric weak solution tending to zero at infinity. Given two symmetric weak solutions $u$ and $v$, we show that if $u$ satisfies the energy inequality $\| \nabla u \|_{L^2 (\Omega)}^2 \le (f,u)$ and $\sup_{x \in \Omega} (|x|+1)|v(x)|$ is sufficiently small, then $u=v$. The proof relies upon a density property for the solenoidal vector field and the Hardy inequality for symmetric functions.
\end{abstract}
\maketitle
\section{Introduction}
Let $\Omega$ be an exterior domain in $\mathbb{R}^2$ with Lipschitz boundary. We study the uniqueness of weak solutions to the stationary Navier-Stokes equation
\begin{equation}
 \label{eq:NS}
 \left\{
  \begin{alignedat}{3}
    - \Delta u + u \cdot \nabla u + \nabla p &= f &\quad&\text{in } \Omega , \\
    \text{div } u &= 0 &\quad&\text{in } \Omega , \\
    u &= 0 &\quad&\text{on } \partial \Omega , \\
    u(x) &\rightarrow 0 &\quad&\text{as } |x| \rightarrow \infty .
  \end{alignedat}
\right.
\end{equation} 
Here $u=(u_1,u_2)$ and $p$ denote, respectively, the unknown velocity and pressure of a viscous incompressible fluid occupying $\Omega$, while $f=(f_1,f_2)$ is a given external force. 

The two-dimensional exterior problem possesses peculiar difficulties. One of the main difficulties stems from the Stokes paradox. It is known that, even if $f=\text{div } F$ with $F=(F_{ij})_{i,j=1,2} \in C_0^\infty (\Omega)$, the Stokes equation
\begin{equation*}
 \left\{
  \begin{alignedat}{3}
    - \Delta u + \nabla p &= f &\quad&\text{in } \Omega , \\
    \text{div } u &= 0 &\quad&\text{in } \Omega , \\
    u &= 0 &\quad&\text{on } \partial \Omega , \\
    u(x) &\rightarrow 0 &\quad&\text{as } |x| \rightarrow \infty ,
  \end{alignedat}
\right.
\end{equation*} 
does not always have a solution. Indeed, it admits a solution only if
\begin{equation*}
\int_{\partial \Omega} (T[u,p]+F) \cdot \nu \,dS =0
\end{equation*} 
where $T[u,p]:=(\partial_i u_j + \partial_j u_i - p \delta_{ij})_{i,j=1,2}$ denotes the stress tensor and $\nu$ is the outer unit normal to $\partial \Omega$, see also \cite{CF,GS,KS2,G1}. Hence the linear approximation is not a useful method in the analysis of the nonlinear problem \eqref{eq:NS} in general. Another difficulty is little information about the asymptotic behavior of Leray's solution in spite of important contributions \cite{GW1,GW2,A}. Leray \cite{L} showed the existence of a weak solution $u$ with finite Dirichlet integral $\int_{\Omega} |\nabla u|^2 \,dx < \infty$ to the problem \eqref{eq:NS}$_{1,2,3}$ with $f=0$, see also \cite{F}. However, it is not known whether his solution of \eqref{eq:NS}$_{1,2,3}$ satisfies \eqref{eq:NS}$_4$ even in a weak sense. This is due to the fact that we cannot control the behavior of the solution $u$ at infinity only from the class $\nabla u \in L^2 (\Omega)$. Owing to these difficulties, the general theory of the existence for \eqref{eq:NS} is not established yet.

By introducing the symmetry, Galdi \cite{G2} and Pileckas-Russo \cite{PR} obtained the existence results concerning \eqref{eq:NS}. We note that the inhomogeneous boundary condition $u=u_*$ on $\partial \Omega$, instead of \eqref{eq:NS}$_3$, is considered in \cite{G2,PR} and \cite{Y2} below, however, we restrict our attention to the problem \eqref{eq:NS}. Assuming that $\Omega$ is symmetric with respect to the coordinate axes $x_1$ and $x_2$:
\begin{equation}
\label{symmetry of domain}
 (x_1,x_2) \in \Omega \ \Rightarrow \ (x_1,-x_2), \ (-x_1,x_2) \in \Omega
\end{equation} 
and $f=(f_1,f_2)$ satisfies the symmetry condition
\begin{equation}
\begin{aligned}
 \label{symmetry of function}
   f_1 (x_1,x_2) = f_1 (x_1,-x_2) = -f_1 (-x_1,x_2), \\
   f_2 (x_1,x_2) = -f_2 (x_1,-x_2) = f_2 (-x_1,x_2),
\end{aligned}
\end{equation}
they proved that the problem \eqref{eq:NS} admits at least one weak solution $u$ with $\nabla u \in L^2 (\Omega)$ and the same symmetry \eqref{symmetry of function}. It was also proved by Galdi \cite{G2} that, due to the symmetry property \eqref{symmetry of function}, the symmetric weak solution $u$ satisfies \eqref{eq:NS}$_4$ in the sense of
\begin{equation}
\label{eq:infinity}
 \lim_{r \rightarrow \infty} \int_0^{2 \pi} |u(r,\theta)|^2 \,d \theta =0,
\end{equation}
see also Russo \cite{R}. Under the stronger symmetry assumption that $\Omega$ satisfies
\begin{equation}
\label{ysymmetry}
  (x_1,x_2) \in \Omega \ \Rightarrow \ (x_1,-x_2), \ (-x_1,x_2), \ (x_2,x_1), \ (-x_2,-x_1) \in \Omega 
\end{equation}
and $f$ satisfies
\begin{equation}
\label{symmetry of function2}
   f_1 (x_2,x_1) = -f_1 (-x_2,-x_1) = f_2 (x_1,x_2)
\end{equation}
as well as \eqref{symmetry of function}, Yamazaki \cite{Y2} showed that if $f$ decays rapidly and is small in a sense, then there exists a weak solution $u$ of \eqref{eq:NS} with $\sup_{x \in \Omega} (|x|+1)|u(x)|$ small and the same symmetry properties \eqref{symmetry of function} and \eqref{symmetry of function2}. To the best of our knowledge, \cite{Y2} is the only literature that provides the existence result of a symmetric weak solution to \eqref{eq:NS} with specific decay rate.

The purpose of this paper is to investigate the uniqueness of weak solutions to \eqref{eq:NS}, which are less symmetric than \eqref{symmetry of function}; to be precise, a weak solution $u=(u_1,u_2)$ is assumed to satisfy the condition that
\begin{equation}
\label{symmetry of function3}
 \begin{split}
  &\mbox{for each } i=1,2 \mbox{ either } u_i (x_1,x_2) = -u_i (x_1,-x_2) \\
  &\mbox{or } u_i (x_1,x_2)= -u_i (-x_1,x_2) \mbox{ holds}.
 \end{split}
\end{equation}
Note that even \eqref{symmetry of function3} is enough to ensure \eqref{eq:infinity}, see \cite{G2,R}. Thus far, there are few results on the uniqueness of weak solutions. Yamazaki \cite{Y2} proved that his solution is unique in the class of weak solutions with $\sup_{x \in \Omega} (|x|+1)|u(x)|$ small as well as symmetry \eqref{symmetry of function} and \eqref{symmetry of function2}, see also \cite{Y1}. We shall show that if $u$ and $v$ are weak solutions of \eqref{eq:NS} with finite Dirichlet integral and symmetry \eqref{symmetry of function3}, $u$ satisfies the energy inequality $\| \nabla u \|_{L^2 (\Omega)}^2 \le (f,u)$ and $\sup_{x \in \Omega} (|x|+1)|v(x)|$ is small, then $u=v$. As an application, our uniqueness theorem, together with the result of Yamazaki \cite{Y2}, describes the asymptotic behavior as $|x| \rightarrow \infty$ of some symmetric weak solutions. Since we consider the homogeneous boundary condition \eqref{eq:NS}$_3$ and in this case it is easy to verify that the symmetric weak solution constructed by Pileckas-Russo \cite{PR} fulfills the energy inequality, we can give information on the asymptotic behavior of their solution such as $|u(x)|=O(|x|^{-1})$ at infinity provided that $f$ satisfies the conditons imposed by \cite{Y2}.

For the proof of our uniqueness theorem, a density property for the solenoidal vector field, together with the Hardy inequality for symmetric functions, plays a crucial role. We shall prove that a function $\psi$ with $\sup_{x \in \Omega} (|x|+1)|\psi (x)| < \infty$ and $\nabla \psi \in L^2 (\Omega)$ can be taken as a test function in the weak form of \eqref{eq:NS}. In two-dimensional exterior domains, we have great difficulty in taking a class of test functions larger than $C_{0,\sigma}^\infty (\Omega)$, while it is relatively easy in $n$-dimensional exterior domains, $n \ge 3$, as we can see in \cite{Mi}. This is due to the lack of information on the class of the nonlinear term $u \cdot \nabla u$. However, thanks to the symmetry property of $u$, the Hardy inequality due to Galdi \cite{G2} (see Lemma \ref{Hardy} below) implies that the term $u \cdot \nabla u$ divided by $|x|+1$ belongs to $L^1 (\Omega)$. With these observations in mind, we shall construct an approximate sequence $\{ \psi_n \}_{n=1}^{\infty} \subset C_{0,\sigma}^\infty (\Omega)$ such that $(|x|+1) \psi_n \rightarrow (|x|+1)\psi$ weakly $*$ in $L^\infty (\Omega)$ as well as $\nabla \psi_n \rightarrow \nabla \psi$ in $L^2 (\Omega)$ as $n \rightarrow \infty$. This density property enables us to take the solution $v$ as a test function in the weak form of \eqref{eq:NS}.

This paper is organized as follows. In Section 2, we shall state the main result on the uniqueness of symmetric weak solutions. After introducing the result of Yamazaki \cite{Y2} precisely, we shall provide a corollary on the asymptotic behavior of a symmetric weak solution. Section 3 is devoted to the proof of the density property mentioned above. The proof of our uniqueness theorem shall be given in Section 4.

\section{Main results}
Before stating our results, we introduce some function spaces. In what follows, we adopt the same symbols for vector and scalar function spaces as long as there is no confusion. For a domain $U \subseteq \mathbb{R}^2$, the space of smooth functions with compact support in $U$ is denoted by $C_0^\infty (U)$ and $C_{0, \sigma}^\infty (U) := \{ \varphi \in C_0^\infty (U) ; \ \text{div } \varphi =0 \ \text{in } U \}$. For $1 \le q \le \infty$, let $L^q (U)$ be the usual Lebesgue space with norm $\| \cdot \|_{q,U}$ and let $W_0^{1,q} (U)$ be the Sobolev space defined by the completion of $C_0^\infty (U)$ in the norm $\| \cdot \|_{1,q,U} := \| \cdot \|_{q,U} + \| \nabla \cdot \|_{q,U}$. We use the abbreviation $\| \cdot \|_q = \| \cdot \|_{q,\Omega}$ for the exterior domain $\Omega$. The homogeneous Sobolev spaces $\dot{H}_0^1 (U)$ and $\dot{H}_{0,\sigma}^1 (U)$ are defined by the completion of $C_0^\infty (U)$ and $C_{0,\sigma}^\infty (U)$, respectively, in the norm $\| \nabla \cdot \|_{2,U}$. If $U$ is bounded, then $\dot{H}_{0,\sigma}^1 (U) = H_{0,\sigma}^1 (U)$ where $H_{0,\sigma}^1 (U) := \overline{C_{0,\sigma}^\infty (U)}^{\| \cdot \|_{1,2,U}}$. The dual space of $\dot{H}_{0,\sigma}^1 (U)$ is denoted by $\dot{H}_{0,\sigma}^{-1} (U)$. By $(\cdot , \cdot)$ we denote various duality pairings.

We also need some symmetry. We say that $\Omega$ is a symmetric exterior domain if $\Omega$ satisfies the condition \eqref{symmetry of domain}. The subspace of $\dot{H}_{0,\sigma}^1 (\Omega)$ consisting of functions with the symmetry property \eqref{symmetry of function3} is denoted by $\dot{H}_{0,\sigma}^{1,S} (\Omega)$.

Our definition of a symmetric weak solution to \eqref{eq:NS} is as follows.

\begin{definition}
Let $\Omega$ be a symmetric exterior domain. Given $f \in \dot{H}_{0,\sigma}^{-1} (\Omega)$, a function $u \in \dot{H}_{0,\sigma}^{1,S} (\Omega)$ is called a symmetric weak solution of \eqref{eq:NS} if $u$ satisfies
\begin{equation}
\label{definition}
  (\nabla u, \nabla \varphi) + (u \cdot \nabla u, \varphi) = (f,\varphi) \quad \text{for all } \varphi \in C_{0,\sigma}^\infty (\Omega).
\end{equation}
\end{definition}

\begin{remark}
If $u \in \dot{H}_{0,\sigma}^{1,S} (\Omega)$, then $u$ automatically satisfies \eqref{eq:infinity}, see Galdi \cite[Lemma 3.2]{G2} and Russo \cite[Theorem 5]{R}.
\end{remark}

\begin{remark}
\label{def}
Our definition of a symmetric weak solution is different from that in \cite{G2,PR}. Let $C_{0,\sigma}^{\infty,S} (\Omega) := \{ \varphi \in C_{0,\sigma}^\infty (\Omega) ;\ \varphi \ \text{satisfies } \eqref{symmetry of function} \}$ and define $H^S$ by the completion of $C_{0,\sigma}^{\infty,S} (\Omega)$ in the norm $\| \nabla \cdot \|_2$. In their definition, for $f \in (H^S)^*$, a function $u \in H^S$ is a symmetric weak solution of \eqref{eq:NS} if $u$ satisfies the weak form \eqref{definition} for all $\varphi \in C_{0,\sigma}^{\infty,S} (\Omega)$. Here $(H^S)^*$ denotes the dual space of $H^S$. Notice that $H^S = \{ u \in \dot{H}_{0,\sigma}^1 (\Omega) ;\ u \ \text{satisfies } \eqref{symmetry of function} \} \subset \dot{H}_{0,\sigma}^{1,S} (\Omega)$ and $\dot{H}_{0,\sigma}^{-1} (\Omega) \subset (H^S)^*$. We can verify that their solutions satisfy \eqref{definition} for all $\varphi \in C_{0,\sigma}^\infty (\Omega)$. Indeed, for  $\varphi \in C_{0,\sigma}^\infty (\Omega)$ we set the function $\varphi^S = (\varphi_1^S, \varphi_2^S) \in C_{0,\sigma}^{\infty,S} (\Omega)$ by
\begin{equation*}
 \begin{aligned}
  \varphi_1^S (x_1,x_2) :&= \frac{1}{4} (\varphi_1 (x_1,x_2)+ \varphi_1 (x_1,-x_2)- \varphi_1 (-x_1,x_2)- \varphi_1 (-x_1,-x_2)), \\
   \varphi_2^S (x_1,x_2) :&= \frac{1}{4} (\varphi_2 (x_1,x_2)- \varphi_2 (x_1,-x_2)+ \varphi_2 (-x_1,x_2)- \varphi_2 (-x_1,-x_2)).
 \end{aligned}
\end{equation*}
Suppose $f \in \dot{H}_{0,\sigma}^{-1} (\Omega)$ with the symmetry property \eqref{symmetry of function} and $u \in H^S$ is a symmetric weak solution in the sense of \cite{G2,PR}. By the symmetry property, direct calculations yield $(\nabla u, \nabla (\varphi - \varphi^S)) = (u \cdot \nabla u, \varphi - \varphi^S) = (f,\varphi - \varphi^S)=0$. Hence
\begin{equation*}
  (\nabla u, \nabla \varphi) + (u \cdot \nabla u, \varphi) = (\nabla u, \nabla \varphi^S) + (u \cdot \nabla u, \varphi^S) = (f,\varphi^S) = (f,\varphi)
\end{equation*}
for all $\varphi \in C_{0,\sigma}^\infty (\Omega)$.
\end{remark}

Now we are in a position to state our main result on the uniqueness of symmetric weak solutions.

\begin{theorem}
\label{theorem}
Let $\Omega$ be a symmetric exterior domain with Lipschitz boundary. Suppose $u,v \in \dot{H}_{0,\sigma}^{1,S} (\Omega)$, having the same symmetry property in \eqref{symmetry of function3}, are symmetric weak solutions of \eqref{eq:NS}. There exists a constant $\delta = \delta (\Omega)$ such that if $u$ satisfies the energy inequality
\begin{equation*}
 \| \nabla u \|_2^2 \le (f,u)
\end{equation*}
and
\begin{equation*}
 \sup_{x \in \Omega} (|x|+1)|v(x)| \le \delta,
\end{equation*}
then $u=v$.
\end{theorem}

\begin{remark}
The existence of a symmetric weak solution was proved by Galdi \cite{G2} and Pileckas-Russo \cite{PR}. It was shown in \cite{PR} that for every $f \in \dot{H}_{0,\sigma}^{-1} (\Omega)$ satisfying \eqref{symmetry of function} there exists a symmetric weak solution $u \in H^S$ of (\ref{eq:NS}), see also Remark \ref{def}. Since we consider the boundary condition $u=0$ on $\partial \Omega$, we can easily verify that the solution constructed by Pileckas-Russo \cite{PR} satisfies the energy inequality. Yamazaki \cite{Y2} obtained a symmetric weak solution $v$ with $\sup_{x \in \Omega} (|x|+1)|v(x)|$ small. For the details of \cite{Y2}, see below.
\end{remark}

\begin{remark}
\label{remark2}
The assumption on the symmetry of weak solutions is closely related to the decay rate of $v$. If $v$ decays faster, that is, $\sup_{x \in \Omega} (|x| +1)^\alpha |v(x)|$ is sufficiently small for some $\alpha>1$, then we can prove the uniqueness without symmetry, see Remark \ref{remark3}.
\end{remark}

\begin{remark}
Our uniqueness theorem is also valid even if we replace $\Omega$ by $\mathbb{R}^2$. This is based on the fact that the Hardy inequality for symmetric functions introduced in Lemma \ref{Hardy} below holds even in $\mathbb{R}^2$. For the existence of a symmetric weak solution $v$ with $\sup_{x \in \mathbb{R}^2} (|x|+1)|v(x)|$ small, see Yamazaki \cite{Y1}.
\end{remark}

\begin{remark}
The same type of uniqueness theorems without symmetry in $n$-dimensional exterior domains, $n \ge 3$, are well known \cite{G1,Mi,KY}.
\end{remark}

We apply our result to deduce the asymptotic behavior of a symmetric weak solution. To this end, we need the following existence result due to Yamazaki \cite{Y2}. Let $\Omega$ be an exterior domain with $C^{2+\mu}$-boundary, $\mu >0$, satisfying \eqref{ysymmetry}. Take $R>0$ so that $\partial \Omega \subset B(0,R):= \{ x \in \mathbb{R}^2 ; |x|<R \}$ and let $\Omega_R := \Omega \cap B(0,R)$. For $q \in [1,\infty)$ and $\alpha >0$, we denote by $\chi (q,\alpha)$ the set of locally integrable functions $f$ on $\Omega$ such that
\begin{equation*}
  \| f \|_{\chi (q,\alpha)} := \frac{R^{\alpha -2/q}}{{\pi}^{1/q}} \| f \|_{q,\Omega_R} + \sup_{r \ge R} \frac{r^{\alpha -2/q}}{{\pi}^{1/q}} \| f \|_{L^q (r \le |x| \le 4r)} < \infty.
\end{equation*}
Then $\chi (q,\alpha)$ is a Banach space and is independent of the choice of $R$ up to equivalent norms. Note also that $\chi (q,\alpha) \subset \chi (s,\alpha)$ if $1 \le s < q < \infty$ and that $\chi (q,\alpha) \subset L^q (\Omega)$ if $\alpha > 2/q$. We especially need the case $q>2$ and $\alpha +1 \in [2,3]$. In such a case, $\chi (q,\alpha +1) \subset L^r (\Omega)$ for all $r \in (1,q]$ and, furthermore, the space $\chi (q,\alpha +1)$ describes the decay of $L^q$-norm in detail.

Assume that the external force $f=(f_1,f_2)$ is represented as
\begin{equation}
\label{form}
 \begin{aligned}
  f_1(x) &= \frac{\partial F}{\partial x_1}(x) + \frac{\partial G}{\partial x_2}(x) + \frac{\partial H}{\partial x_2}(x) , \\ 
  f_2(x) &= -\frac{\partial F}{\partial x_2}(x) + \frac{\partial G}{\partial x_1}(x) - \frac{\partial H}{\partial x_1}(x)
 \end{aligned}
\end{equation}
with scalar-valued functions $F(x)$, $G(x)$ and $H(x)$ satisfying the symmetry conditions
\begin{equation}
 \label{F}
 \left\{
  \begin{alignedat}{3}
    &F(x_1,x_2)=F(x_1,-x_2)=F(-x_1,x_2), \\
    &F(x_1,x_2)=-F(x_2,x_1)=-F(-x_2,-x_1),
  \end{alignedat}
\right.
\end{equation} 
\begin{equation}
\label{G}
 \left\{
  \begin{alignedat}{3}
    &G(x_1,x_2)=-G(x_1,-x_2)=-G(-x_1,x_2), \\
    &G(x_1,x_2)=G(x_2,x_1)=G(-x_2,-x_1), 
  \end{alignedat}
\right.
\end{equation} 
\begin{equation}
\label{H}
 \left\{
  \begin{alignedat}{3}
    &H(x_1,x_2)=-H(x_1,-x_2)=-H(-x_1,x_2), \\
    &H(x_1,x_2)=-H(x_2,x_1)=-H(-x_2,-x_1).
  \end{alignedat}
\right.
\end{equation}
Notice that $f$ satisfies the symmetry properties \eqref{symmetry of function} and \eqref{symmetry of function2}. Yamazaki \cite{Y2} proved that for $q>2$ and $\alpha \in [1,2]$ there exists a constant $\beta =\beta (\Omega,q,\alpha)$ such that if
\begin{equation}
\label{smallness}
  \| F \|_{\chi (q,\alpha+1)} +  \| G \|_{\chi (q,\alpha+1)} +  \| H \|_{\chi (q,\alpha+1)} \le \beta ,
\end{equation}
then the problem \eqref{eq:NS} admits a unique solution $u$ with $\nabla u \in \chi (q,\alpha+1)$ and $\sup_{x \in \Omega} (|x|+1)^\alpha |u(x)| < \infty$ subject to the estimate
\begin{equation}
\label{smallness2}
  \begin{aligned}
  \sup_{x \in \Omega} (|x|+1)^\alpha |&u(x)| + \| \nabla u \|_{\chi (q,\alpha+1)} \\ &\le \gamma (\| F \|_{\chi (q,\alpha+1)} +  \| G \|_{\chi (q,\alpha+1)} +  \| H \|_{\chi (q,\alpha+1)})
  \end{aligned}
\end{equation}
with $\gamma = \gamma (\Omega,q,\alpha)$. The solution $u$ also satisfies the symmetry properties \eqref{symmetry of function} and \eqref{symmetry of function2}.

Observe that we may assume the external force $f$ is given in the form \eqref{form} without loss of generality. Indeed, if $f=\text{div } \Psi = \left( \sum_{i=1}^2 \partial_i \Psi_{ij} \right)_{j=1,2}$ with $\Psi =\{ \Psi_{ij} \}_{i,j=1,2}$, then we put
\begin{align*}
   \Phi &= \frac{1}{2} (\Psi_{11} + \Psi_{22}), \quad F = \frac{1}{2} (\Psi_{11} - \Psi_{22}), \\
   G &= \frac{1}{2} (\Psi_{12} + \Psi_{21}), \quad H = \frac{1}{2} (-\Psi_{12} + \Psi_{21}),
\end{align*}
to deduce that $f$ is represented as \eqref{form} by absorbing the term $\nabla \Phi$ into $\nabla p$. Note also that, by the properties of $\chi (q,\alpha+1)$, we have $F,G,H,\nabla u \in L^r (\Omega)$ for every $r \in (1,q]$. Since $q>2$, it follows that $f \in \dot{H}_{0,\sigma}^{-1} (\Omega)$ and $u \in \dot{H}_{0,\sigma}^{1,S} (\Omega)$ in particular.

As a consequence of Theorem \ref{theorem} and the result of Yamazaki \cite{Y2} mentioned above, we derive the following assertion.

\begin{corollary}
Let $q \in (2,\infty)$. Assume that $\Omega$ is an exterior domain with $C^{2+\mu}$-boundary, $\mu >0$, satisfying \eqref{ysymmetry} and the external force $f$ is given in the form \eqref{form} with $F$, $G$ and $H$ satisfying the conditions \eqref{F}, \eqref{G} and \eqref{H} respectively. Suppose $u \in \dot{H}_{0,\sigma}^{1,S} (\Omega)$ is a symmetric weak solution of \eqref{eq:NS} with the energy inequality $\| \nabla u \|_2^2 \le (f,u)$. If
\begin{equation*}
 \| F \|_{\chi (q,2)} +  \| G \|_{\chi (q,2)} +  \| H \|_{\chi (q,2)} \le \min \{ \beta, \gamma^{-1} \delta \}
\end{equation*}
where $\beta =\beta (\Omega,q,1)$, $\gamma = \gamma (\Omega,q,1)$ and $\delta = \delta (\Omega)$ are the constants, respectively, in \eqref{smallness}, \eqref{smallness2} and Theorem \ref{theorem}, then
\begin{equation*}
 (|x|+1)|u(x)| \in L^\infty (\Omega) \quad \text{and} \quad \nabla u \in L^r (\Omega) \ \ \text{for every } \ r \in (1,q] .
\end{equation*}
\end{corollary}

\section{Density property}
In this section we prove the density property for the solenoidal vector field. The main result in this section is the following proposition. 
\begin{proposition}
\label{exterior}
Let $\Omega$ be an exterior domain with Lipschitz boundary. For every $v \in \dot{H}_{0,\sigma}^1 (\Omega) $ with
\begin{equation*}
 \sup_{x \in \Omega} (|x|+1) |v(x)| < \infty,
\end{equation*}
there exists a sequence $\{ v_n \}_{n=1}^{\infty} \subset C_{0,\sigma}^\infty (\Omega)$ such that
\begin{equation*}
 \begin{aligned}
 &\nabla v_n \rightarrow \nabla v &\quad&\text{in } L^2 (\Omega) , \\
 &(|x|+1) v_n \rightarrow (|x|+1)v &\quad&\text{weakly } * \ \text{in } L^\infty (\Omega)
 \end{aligned}
\end{equation*}
as $n \rightarrow \infty$.
\end{proposition}
For the proof of this proposition, we show the corresponding density property in the whole plane $\mathbb{R}^2$ and bounded domains. Based on the analysis in $\mathbb{R}^2$ and bounded domains, we can prove the density property in exterior domains above. It should be emphasized that we need no symmetry in this section.

In what follows, we denote by $C$ various constants and, in particular, $C=C(\cdot,\cdots,\cdot)$ denotes constants depending only on the quantities in parentheses. We first introduce the Bogovski operator.

\begin{lemma}[\cite{B,BS,G1}]
\label{div}
Let $D$ be a bounded domain in $\mathbb{R}^n$, $n \ge 2$, with Lipschitz boundary and $1<q<\infty$.\\
$(\mathrm{i})$ There exists a linear operator $B_D : C_{0}^{\infty} (D) \rightarrow  C_{0}^{\infty} (D)^n$ such that
\begin{equation*}
 \| \nabla B_D f \|_{q,D} \le C \| f \|_{q,D}
\end{equation*}
with $C=C(n,q,D)$ independent of $f$ and that
\begin{equation*}
 \text{div } B_D f = f \quad \text{if}  \quad \int_{D} f \,dx = 0.
\end{equation*}
Furthermore, by continuity, $B_D$ is uniquely extended to a bounded linear operator from $L^q (D)$ to $W_0^{1,q} (D)^n$.\\
$(\mathrm{ii})$ Let $y \in \mathbb{R}^n$, $t \in \mathbb{R} \setminus \{ 0 \}$ and
\begin{equation*}
 D_t := \{ (1-t)y+tx :  x \in D \}.
\end{equation*}
Then the constant $C=C(n,q,D_t)$ associated with the operator $B_{D_t}$ is independent of $y$ and $t$.
\end{lemma}

The next lemma concerns the density property in the whole plane $\mathbb{R}^2$.

\begin{lemma}
\label{plane}
For every $v \in \dot{H}_{0,\sigma}^1 (\mathbb{R}^2)$ with
\begin{equation*}
 \sup_{x \in \mathbb{R}^2} (|x| +1)|v(x)| < \infty,
\end{equation*}
there exists a sequence $\{ v_n \}_{n=1}^{\infty} \subset C_{0,\sigma}^\infty (\mathbb{R}^2)$ such that
\begin{equation*}
 \begin{aligned}
  &\nabla v_n \rightarrow \nabla v &\quad&\text{in } L^2 (\mathbb{R}^2), \\
  &(|x|+1) v_n \rightarrow (|x|+1)v &\quad&\text{weakly } * \ \text{in } L^\infty (\mathbb{R}^2)
 \end{aligned}
\end{equation*}
as $n \rightarrow \infty$.
\end{lemma}

\begin{proof}
Choose a cutoff function $\psi \in C_0^{\infty} (\mathbb{R}^2)$ such that $0 \le \psi (x) \le 1$, $\psi (x)=1$ for $|x| \le 1$ and $\psi (x)=0$ for $|x| \ge 2$, and set $\psi_m (x) := \psi (x/m)$. We put
\begin{equation*}
 v_\epsilon := J_\epsilon * v, \qquad v_{\epsilon,m} := \psi_m v_\epsilon - B_m [\nabla \psi_m \cdot v_\epsilon],
\end{equation*}
where $J_{\epsilon}$ is the Friedrichs mollifier and $B_m$ is the operator introduced in Lemma \ref{div} for the bounded domain $E_m := \{ m/2<|x|<3m \}$. Since $\text{div } v_\epsilon =0$ in $\mathbb{R}^2$ and $\int_{E_m} \nabla \psi_m \cdot v_\epsilon \,dx = 0$, we can verify that $v_{\epsilon,m} \in C_{0,\sigma}^{\infty} (\mathbb{R}^2)$ for all $m=1,2,\ldots$ and $\epsilon >0$.

Let $M:=\sup_{x \in \mathbb{R}^2} (|x| +1)|v(x)|$. We show the estimate
\begin{equation}
 \label{uniform}
  \sup_{x \in \mathbb{R}^2} (|x|+1) |v_{\epsilon,m} (x)| \le CM
\end{equation}
with $C$ independent of $m$ and $\epsilon$. We have
\begin{equation*}
  (|x|+1)|v_\epsilon| \le  I_1 + I_2
\end{equation*}
where
\begin{equation*}
 \begin{aligned}
  I_1 (x) &:= \int_{\mathbb{R}^2} J_\epsilon (x-y) ||x|-|y|| |v(y)| \,dy , \\
  I_2 (x) &:= \int_{\mathbb{R}^2} J_\epsilon (x-y) (|y|+1) |v(y)| \,dy .
 \end{aligned}
\end{equation*}
We see
\begin{equation*}
  I_2 \le \sup_{y \in \mathbb{R}^2} (|y|+1)|v(y)| \int_{\mathbb{R}^2} J_\epsilon (x-y) \,dy = M.
\end{equation*}
Recall that $J_\epsilon (x-y) =0$ for $|x-y| \ge \epsilon$, and for $|x-y| \le \epsilon \le 1$ there holds
\begin{equation*}
  ||x|-|y|| \le |x-y| \le \epsilon \le 1.
\end{equation*}
Hence
\begin{equation*}
 I_1 \le \sup_{y \in \mathbb{R}^2} |v(y)| \int_{\mathbb{R}^2} J_\epsilon (x-y) \,dy \le M.
\end{equation*}
Therefore we obtain the estimate
\begin{equation}
 \label{v_epsilon}
 (|x|+1) |v_\epsilon| \le 2M.
\end{equation}
On the other hand, since $\text{supp } B_m [\nabla \psi_m \cdot v_\epsilon]$ is contained in $E_m$, it follows from the Gagliardo-Nirenberg inequality and the Sobolev embedding that
\begin{align*}
  (|x|+1)|B_m [\nabla \psi_m \cdot v_\epsilon]| &\le Cm|B_m [\nabla \psi_m \cdot v_\epsilon]| \\
  &\le Cm \| B_m [\nabla \psi_m \cdot v_\epsilon] \|_{4,E_m}^{1/2} \| \nabla B_m [\nabla \psi_m \cdot v_\epsilon] \|_{4,E_m}^{1/2} \\
  &\le Cm \| \nabla B_m [\nabla \psi_m \cdot v_\epsilon] \|_{4/3,E_m}^{1/2} \| \nabla \psi_m \cdot v_\epsilon \|_{4,E_m}^{1/2} \\
  &\le Cm \| \nabla \psi_m \cdot v_\epsilon \|_{4/3,E_m}^{1/2} \| \nabla \psi_m \cdot v_\epsilon \|_{4,E_m}^{1/2} .
\end{align*}
Note that the constants $C$ above are independent of $m$ and $\epsilon$, due to Lemma \ref{div}$(\mathrm{ii})$. Since $|v_\epsilon| \le CM/m$ on $E_m$ and $|\nabla \psi_m| \le C/m$ for some constants $C$ independent of $m$ and $\epsilon$, direct calculations yield
\begin{equation*}
 \| \nabla \psi_m \cdot v_\epsilon \|_{4/3,E_m}^{1/2} \le CM^{1/2} m^{-1/4}, \quad \| \nabla \psi_m \cdot v_\epsilon \|_{4,E_m}^{1/2} \le CM^{1/2} m^{-3/4}.
\end{equation*}
Thus we derive
\begin{equation}
\label{B_m}
 (|x|+1)|B_m [\nabla \psi_m \cdot v_\epsilon]| \le CM
\end{equation}
with $C$ independent of $m$ and $\epsilon$. The uniform estimate \eqref{uniform} follows from \eqref{v_epsilon} and \eqref{B_m}.

Next, in view of $\text{supp } \nabla \psi_m \subset E_m$ and $|v_\epsilon| \le CM/m$ on $E_m$, we have
\begin{align*}
  &\| \nabla v_{\epsilon,m} - \nabla v_\epsilon \|_{2,\mathbb{R}^2} \\
  \le &\| (\psi_m -1) \nabla v_\epsilon \|_{2,\mathbb{R}^2} + \| (\nabla \psi_m) v_\epsilon \|_{2,E_m} + \| \nabla B_m [\nabla \psi_m \cdot v_\epsilon] \|_{2,E_m} \\
  \le &\| (\psi_m -1) \nabla v_\epsilon \|_{2,\mathbb{R}^2} + \| (\nabla \psi_m) v_\epsilon \|_{2,E_m} + C \| \nabla \psi_m \cdot v_\epsilon \|_{2,E_m} \\
  \le &\| (\psi_m -1) \nabla v_\epsilon \|_{2,\mathbb{R}^2} + CMm^{-1} \\
  \rightarrow &\ 0 \quad \text{as } \ m \rightarrow \infty.
\end{align*}
Here we have used Lemma \ref{div}$(\mathrm{ii})$. Furthermore, the class of $\nabla v$ implies
\begin{equation*}
  \nabla v_\epsilon \rightarrow \nabla v \quad \text{in } L^2 (\mathbb{R}^2) \quad \text{as } \  \epsilon \downarrow 0.
\end{equation*}

From the arguments above, $v$ is an accumulation point of the two-parameters family $\{ v_{\epsilon,m} \}_{\epsilon >0, m \in \mathbb{N}} \subset C_{0,\sigma}^\infty (\mathbb{R}^2)$ in $\dot{H}_{0,\sigma}^1 (\mathbb{R}^2)$, that is, we can take a subsequence  $\{ v_{\epsilon_j,m_j} \}_{j=1}^{\infty}$ such that $\| \nabla v_{\epsilon_j,m_j} - \nabla v \|_{2,\mathbb{R}^2} \le \frac{1}{j}$. We conclude from the uniform estimate \eqref{uniform} that there exists a subsequence  $\{ v_n \}_{n=1}^{\infty} \subset \{ v_{\epsilon_j,m_j} \}_{j=1}^{\infty}$ satisfying the desired density property.
\end{proof}

In order to establish the density property in a bounded domain $D$, we need two lemmas. We first construct an approximate sequence when $D$ is star-shaped, by following the argument due to Masuda \cite[Proposition 1]{Ma}. Recall that $D$ is star-shaped with respect to some point $x \in D$ if $\lambda^{-1} (\overline{D} -x) \subset D-x$ for all $\lambda >1$. By a translation we may assume that $\lambda^{-1} \overline{D} \subset D$ for all $\lambda >1$ if $D$ is star-shaped.

\begin{lemma}
\label{star-shaped}
Let $D$ be a star-shaped bounded domain with Lipschitz boundary. For every $v \in H_{0,\sigma}^1 (D) \cap L^\infty (D)$, there exists a sequence $\{ v_n \}_{n=1}^\infty \subset C_{0,\sigma}^\infty (D)$ such that
\begin{equation*}
 \begin{aligned}
 &\nabla v_n \rightarrow \nabla v &\quad&\text{in } L^2 (D), \\
 &(|x|+1)v_n \rightarrow (|x|+1)v &\quad&\text{weakly } * \ \text{in } L^\infty (D)
 \end{aligned}
\end{equation*}
as $n \rightarrow \infty$.
\end{lemma}

\begin{proof}
Let $\tilde{v}$ be the zero extension of $v$, that is, $\tilde{v} (x)=v(x)$ if $x \in D$ and  $\tilde{v} (x)=0$ if $x \in \mathbb{R}^2 \setminus D$. For $\lambda >1$ and small $\epsilon >0$, we set
\begin{equation*}
 v_\lambda (x):= \tilde{v} (\lambda x) , \quad v_{\lambda, \epsilon} := J_\epsilon * v_\lambda
\end{equation*}
where $J_{\epsilon}$ is the Friedrichs mollifier. It follows from $\text{supp } v_\lambda \subset {\lambda}^{-1} \overline{D} \subset D$ that $\text{supp } v_{\lambda, \epsilon} \subset D$. Thus $v_{\lambda, \epsilon} \in C_{0,\sigma}^\infty (D)$ for all $\lambda >1$ and small $\epsilon >0$. We can also verify that $\nabla v_{\lambda,\epsilon} \rightarrow \nabla v_\lambda$ in $L^2 (D)$ as $\epsilon \downarrow 0$. Since $C_{0,\sigma}^\infty (D)$ is dense in $H_{0,\sigma}^1 (D)$, for each $\kappa >0$ there exists a function $\psi \in C_{0,\sigma}^\infty (D)$ such that $\| \nabla v - \nabla \psi \|_{2,D} =\| \nabla v_\lambda - \nabla \psi_\lambda \|_{2,D} < \kappa /3$. In addition, by the uniform continuity of $\nabla \psi$, there holds $\| \nabla \psi_\lambda - \nabla \psi \|_{2,D} < \kappa /3$ provided $1<\lambda \le 1+\delta$ for sufficiently small $\delta>0$. Hence, for $1<\lambda \le 1+\delta$, we deduce
\begin{equation*}
 \begin{aligned}
  \| \nabla v_\lambda - \nabla v \|_{2,D} &\le \| \nabla v_\lambda - \nabla \psi_\lambda \|_{2,D} + \| \nabla \psi_\lambda - \nabla \psi \|_{2,D} + \| \nabla \psi - \nabla v \|_{2,D} \\
  &< \frac{\kappa}{3} + \frac{\kappa}{3} + \frac{\kappa}{3} \\
  &= \kappa.
 \end{aligned}
\end{equation*}
Furthermore, we have $\| v_\lambda \|_{\infty,\mathbb{R}^2} =\| v \|_{\infty,D}$, which gives the estimate
\begin{equation*}
 \| v_{\lambda,\epsilon} \|_{\infty,D} \le \| v \|_{\infty,D}.
\end{equation*}

Since $v$ is an accumulation point of the family $\{ v_{\lambda,\epsilon} \}_{\lambda >1, \epsilon >0} \subset C_{0,\sigma}^\infty (D)$ in $H_{0,\sigma}^1 (D)$ and the uniform estimate above holds, we can take a subsequence $\{ v_n \}_{n=1}^{\infty}$ such that
\begin{equation*}
 \nabla v_n \rightarrow \nabla v \quad \text{in } L^2 (D) \ \ \text{as} \ \ n \rightarrow \infty \quad \text{and} \quad \| v_n \|_{\infty,D} \le \| v \|_{\infty,D}.
\end{equation*}
The estimate yields
\begin{equation}
 \label{uniform2}
 \| (|x|+1)v_n \|_{\infty,D} \le C \| v \|_{\infty,D}
\end{equation}
with $C=C(D)$. For $\varphi \in C_0^\infty (D)$ we employ the Poincar\'{e} inequality to obtain
\begin{equation}
\label{Poincare}
 \begin{aligned}
  ((|x|+1)v_n - (|x|+1)v, \varphi) &\le \| v_n -v \|_{2,D}  \| (|x|+1) \varphi \|_{2,D} \\
  &\le C \| \nabla v_n - \nabla v \|_{2,D}  \| (|x|+1) \varphi \|_{2,D} \\
  &\rightarrow 0 \quad \text{as } \ n \rightarrow \infty .
 \end{aligned}
\end{equation}
Since $C_0^\infty (D)$ is dense in $L^1 (D)$, it follows from \eqref{uniform2} and \eqref{Poincare} that
\begin{equation*}
 (|x|+1)v_n \rightarrow (|x|+1)v \quad \text{weakly } * \ \text{in} \ L^\infty (D)
\end{equation*}
as $n \rightarrow \infty$. The proof is complete.
\end{proof}

Next, we employ a localization procedure which is similar to Abe-Giga \cite[Lemma 6.2]{AG}.
\begin{lemma}
\label{localization}
Let $D$ be a bounded domain with Lipschitz boundary. Suppose $\{ G_m \}_{m=1}^N$ is an open covering of $\overline{D}$ and $D_m := D \cap G_m$. Then there exists a family of bounded linear operators $\{ T_m \}_{m=1}^N$ from $H_{0,\sigma}^1 (D) \cap L^\infty (D)$ to $H_{0,\sigma}^1 (D_m) \cap L^\infty (D_m)$ satisfying $v=\sum_{m=1}^N T_m v$.
\end{lemma}

\begin{proof}
Following Abe-Giga \cite[Lemma 6.2]{AG}, we give the proof by induction with respect to $N$. If $N=1$, the assertion is obvious.

Assume that the assertion is valid for $N$. Set
\begin{equation*}
 U := \bigcup_{m=2}^{N+1} D_m, \quad V := \bigcup_{m=2}^{N+1} G_m, \quad E:=D_1 \cap U.
\end{equation*}
Then $D=D_1 \cup U$ and $\{ G_1,V \}$ is a covering of $\overline{D}$. Let $\{ \xi_1, \xi_2 \}$ be a smooth partition of unity of $D$ associated with $\{ G_1,V \}$, that is, $0 \le \xi_i \le 1$ $(i=1,2)$, $\text{supp } \xi_1 \subset G_1$, $\text{supp } \xi_2 \subset V$ and $\xi_1 + \xi_2 =1$ on $D$. For $v \in H_{0,\sigma}^1 (D) \cap L^\infty (D)$ we define the operator $T_1$ by
\begin{equation*}
 T_1 v := \xi_1 v - B_E [\nabla \xi_1 \cdot v]
\end{equation*}
where $B_E$ is the Bogovski operator defined by Lemma \ref{div} for $E$. In the case where $E$ is the union of disjoint Lipschitz domains, for instance, $E_1$ and $E_2$, we have only to replace the term $B_E [\nabla \xi_1 \cdot v]$ above by $ \sum_{i=1}^2 B_{E_i} [\nabla \xi_1 \cdot v]$. Since $\nabla \xi_1 =0$ in $D_1 \setminus E$, we have
\begin{equation*}
 \int_{E} \nabla \xi_1 \cdot v \,dx = \int_{D_1} \nabla \xi_1 \cdot v \,dx =0 .
\end{equation*}
Hence Lemma \ref{div}($\mathrm{i}$) and $\nabla \xi_1 \cdot v \in L^\infty (E)$ imply $\text{div } T_1 v=0$ in $D_1$ and $B_E [\nabla \xi_1 \cdot v] \in W_0^{1,q} (E)$ for all $1 < q < \infty$. Using the Sobolev embedding, Lemma \ref{div}$(\mathrm{i})$ and the Poincar\'{e} inequality, for $q>2$ we obtain
\begin{equation}
 \label{Bogovski estimate}
  \| B_E [\nabla \xi_1 \cdot v] \|_{\infty,E} \le C\| B_E [\nabla \xi_1 \cdot v] \|_{1,q,E} \le C\| \nabla \xi_1 \cdot v \|_{q,E} \le C\| v \|_{\infty,D}.
\end{equation}
This estimate, together with $\| T_1 v \|_{1,2,D_1} \le C \| v \|_{1,2,D}$, shows that $T_1$ is a bounded linear operator from $H_{0,\sigma}^1 (D) \cap L^\infty (D)$ to $H_{0,\sigma}^1 (D_1) \cap L^\infty (D_1)$. On the other hand, we put
\begin{equation*}
 T_U v := \xi_2 v - B_E [\nabla \xi_2 \cdot v].
\end{equation*}
The same argument as above yields that $T_U$ is a bounded linear operator from $H_{0,\sigma}^1 (D) \cap L^\infty (D)$ to $H_{0,\sigma}^1 (U) \cap L^\infty (U)$. Furthermore
\begin{equation*}
 v= T_1 v + T_U v.
\end{equation*}
Since $U$ is covered by $\{ G_m \}_{m=2}^{N+1}$, by the induction assumption there exists a family of bounded linear operators $\{ \widehat{T}_m \}_{m=2}^{N+1}$ from $H_{0,\sigma}^1 (U) \cap L^\infty (U)$ to $H_{0,\sigma}^1 (D_m) \cap L^\infty (D_m)$ satisfying $u=\sum_{m=2}^{N+1} \widehat{T}_m u$ for $u \in H_{0,\sigma}^1 (U) \cap L^\infty (U)$. Setting
\begin{equation*}
 T_1 := T_1, \quad T_m := \widehat{T}_m \cdot T_U \quad (m=2, \ldots ,N+1),
\end{equation*}
we conclude that $\{ T_m \}_{m=1}^{N+1}$ is a family of bounded linear operators from $H_{0,\sigma}^1 (D) \cap L^\infty (D)$ to $H_{0,\sigma}^1 (D_m) \cap L^\infty (D_m)$ satisfying $v=\sum_{m=1}^{N+1} T_m v$.
\end{proof}

Collecting Lemmas \ref{star-shaped} and \ref{localization}, we can construct an approximate sequence in general bounded domains.

\begin{lemma}
\label{bounded}
The assertion in Lemma \ref{star-shaped} is also valid when $D$ is a bounded domain with Lipschitz boundary.
\end{lemma}

\begin{proof}
It is well known that, by the assumption on the boundary $\partial D$, there exists an open covering $\{ G_m \}_{m=1}^N$ of $\overline{D}$ such that $D_m =D \cap G_m$ $(m=1, \ldots, N)$ are star-shaped bounded domains with Lipschitz boundary with respect to some open balls in $D_m$. Let $\{ T_m \}_{m=1}^N$ be the family of bounded linear operators introduced in Lemma \ref{localization}. For $m=1, \ldots , N$, we put $v_m := T_m v$. Then $v_m \in H_{0,\sigma}^1 (D_m) \cap L^\infty (D_m)$ and $v=\sum_{m=1}^N v_m$. Since $D_m$ are star-shaped, for each $m=1,\ldots ,N$ we can take by the proof of Lemma \ref{star-shaped} a sequence $\{ v_{m,n} \}_{n=1}^\infty \subset C_{0,\sigma}^\infty (D_m)$ such that
\begin{equation*}
 \begin{aligned}
 &\nabla v_{m,n} \rightarrow \nabla v_m \quad \text{in } L^2 (D_m) \ \ \text{as } \ n \rightarrow \infty, \\
 &\| (|x|+1)v_{m,n} \|_{\infty,D_m} \le C \| v_m \|_{\infty,D_m}
  \end{aligned}
\end{equation*}
with $C$ independent of $n$. We denote the zero extension of $v_{m,n}$ to $D \setminus D_m$ by $v_{m,n}$ itself for simplicity, and set $v_n := \sum_{m=1}^N v_{m,n}$. Then we derive
\begin{equation*}
  \| \nabla v_n - \nabla v \|_{2,D} \le \sum_{m=1}^N \| \nabla v_{m,n}  - \nabla v_m \|_{2,D_m} \rightarrow 0 \quad \text{as } \ n \rightarrow \infty .
\end{equation*}
Since $T_m$ is a bounded linear operator from $H_{0,\sigma}^1 (D) \cap L^\infty (D)$ to $H_{0,\sigma}^1 (D_m) \cap L^\infty (D_m)$, we have
\begin{align*}
  \| (|x|+1)v_n \|_{\infty,D} &\le \sum_{m=1}^N \| (|x|+1)v_{m,n} \|_{\infty,D_m} \\
 &\le \sum_{m=1}^N C \| v_m \|_{\infty,D_m} \\
 &\le C \| v \|_{H_{0,\sigma}^1 (D) \cap L^\infty (D)}
\end{align*}
with $C$ independent of $n$. This estimate, together with the same calculation as \eqref{Poincare} and the density property of $C_0^\infty (D)$ in $L^1 (D)$, yields
\begin{equation*}
 (|x|+1)v_n \rightarrow (|x|+1)v \quad \text{weakly } * \ \text{in } L^\infty (D)
\end{equation*}
as $n \rightarrow \infty$, and the result follows.
\end{proof}

Using Lemmas \ref{plane} and \ref{bounded}, we can prove Proposition \ref{exterior}. For the proof, we follow Kozono-Sohr \cite[Theorem 2]{KS}.

\begin{proof}[Proof of Proposition \ref{exterior}]
Let $M:=\sup_{x \in \Omega} (|x|+1) |v(x)|$ and take $R>0$ so that $\partial \Omega \subset B(0,R)$. We define a function $\tilde{v}$ by the zero extension of $v$. Then $\tilde{v} \in \dot{H}_{0,\sigma}^1 (\mathbb{R}^2)$ with $\sup_{x \in \mathbb{R}^2} (|x| +1)|\tilde{v} (x)| = M$. In view of Lemma \ref{plane}, there exists a sequence $\{ \tilde{v}_n \}_{n=1}^{\infty} \subset C_{0,\sigma}^\infty (\mathbb{R}^2)$ such that
\begin{equation}
 \label{in R2}
 \begin{aligned}
  &\nabla \tilde{v}_n \rightarrow \nabla \tilde{v} &\quad&\text{in } L^2 (\mathbb{R}^2), \\
  &(|x|+1) \tilde{v}_n \rightarrow (|x|+1) \tilde{v} &\quad&\text{weakly } * \ \text{in } L^{\infty} (\mathbb{R}^2)
 \end{aligned}
\end{equation}
as $n \rightarrow \infty$. In addition, we observe that $\nabla \tilde{v}_n \rightarrow \nabla \tilde{v}$ in $L^2 (\mathbb{R}^2)$ implies $\tilde{v}_n \rightarrow \tilde{v}$ in $L^2 (\Omega_R)$. Indeed, by the definition of $\tilde{v}$ and the construction of $\tilde{v}_n$ in the proof of Lemma \ref{plane}, we may assume $\tilde{v}_n - \tilde{v} =0$ in some open ball contained in $\mathbb{R}^2 \setminus \Omega$. Hence we employ the Poincar\'{e} inequality to deduce
\begin{equation}
 \label{loc}
 \| \tilde{v}_n - \tilde{v} \|_{2,\Omega_R} \le C \| \nabla \tilde{v}_n - \nabla \tilde{v} \|_{2,\Omega_R} \rightarrow 0 \quad \text{as} \ \ n \rightarrow \infty.
\end{equation}

Let $\zeta \in C^\infty (\mathbb{R}^2)$ be a cutoff function such that $0 \le \zeta (x) \le 1$, $\zeta (x) = 1$ for $|x| \ge R$ and $\zeta (x) = 0$ in the neighbourhood of $\partial \Omega$. Put $w_n := B_{\Omega_R} [\nabla \zeta \cdot \tilde{v}_n]$ and $w := B_{\Omega_R} [\nabla \zeta \cdot \tilde{v}]$ where $B_{\Omega_R}$ is the operator defined by Lemma \ref{div} for $\Omega_R = \Omega \cap B(0,R)$. Since $\nabla \zeta \cdot \tilde{v}_n \in C_0^\infty (\Omega_R)$ and $\int_{\Omega_R} \nabla \zeta \cdot \tilde{v}_n \,dx=0$, we deduce $w_n \in C_0^\infty (\Omega_R)$ and $\text{div } w_n = \nabla \zeta \cdot \tilde{v}_n$ in $\Omega_R$. Similarly, it follows from $\nabla \zeta \cdot \tilde{v} \in L^\infty (\Omega_R)$ and $\int_{\Omega_R} \nabla \zeta \cdot \tilde{v} \,dx=0$ that $w \in W_0^{1,q} (\Omega_R)$ $(1<q<\infty)$ satisfies $\text{div } w= \nabla \zeta \cdot \tilde{v}$ in $\Omega_R$. Furthermore, by Lemma \ref{div}$(\mathrm{i})$ and \eqref{loc}, we obtain
\begin{equation}
 \label{w_n1}
 \| \nabla w_n - \nabla w \|_{2,\Omega_R} \le C \| \nabla \zeta \cdot (\tilde{v}_n - \tilde{v}) \|_{2,\Omega_R} \rightarrow 0
\end{equation}
as $n \rightarrow \infty$. From the proof of Lemma \ref{plane}, we may assume $\| \tilde{v}_n \|_{\infty,\mathbb{R}^2} \le CM$ with $C$ independent of $n$. Thus the same calculation as \eqref{Bogovski estimate} yields
\begin{equation*}
 \| (|x|+1) w_n \|_{\infty,\Omega_R} \le CM
\end{equation*}
with $C=C(R)$. This estimate, together with the same calculation as \eqref{Poincare} and the density property of $C_0^\infty (\Omega_R)$ in $L^1 (\Omega_R)$, leads us to
\begin{equation}
 \label{w_n2}
 (|x|+1) w_n \rightarrow (|x|+1) w \quad \text{weakly } * \ \text{in } L^\infty (\Omega_R) \quad (n \rightarrow \infty).
\end{equation}

We also set $u:=(1-\zeta) \tilde{v} +w$. Then $u \in H_{0,\sigma}^1 (\Omega_R) \cap L^\infty (\Omega_R)$, and hence, according to Lemma \ref{bounded}, we can take a sequence $\{ u_n \}_{n=1}^\infty \subset C_{0,\sigma}^\infty (\Omega_R)$ such that
\begin{equation}
 \label{u_n}
 \begin{aligned}
  &\nabla u_n \rightarrow \nabla u &\quad&\text{in } L^2 (\Omega_R) , \\
  &(|x|+1) u_n \rightarrow (|x|+1)u &\quad&\text{weakly } * \ \text{in } L^\infty (\Omega_R)
 \end{aligned}
\end{equation}
as $n \rightarrow \infty$.

Now we define the sequence $\{ v_n \}_{n=1}^\infty$ by
\begin{equation*}
 v_n := \zeta \tilde{v}_n - \tilde{w}_n + \tilde{u}_n 
\end{equation*}
where $\tilde{w}_n$ and $\tilde{u}_n$ denote the zero extension of $w_n$ and $u_n$ respectively. Then $v_n \in C_{0,\sigma}^\infty (\Omega)$ for all $n=1,2,\ldots$. Since $v(x)=\zeta (x) \tilde{v} (x)-w(x) + u(x)$ for $x \in \Omega$, the properties \eqref{in R2}, \eqref{loc}, \eqref{w_n1}, \eqref{w_n2} and \eqref{u_n} yield
\begin{equation*}
 \begin{aligned}
 &\nabla v_n \rightarrow \nabla v &\quad&\text{in } L^2 (\Omega) , \\
 &(|x|+1)v_n \rightarrow (|x|+1)v &\quad&\text{weakly } * \ \text{in } L^\infty (\Omega)
 \end{aligned}
\end{equation*}
as $n \rightarrow \infty$.
\end{proof}

\begin{remark}
\label{alpha}
In the case $\sup_{x \in \Omega} (|x|+1)^\alpha |v(x)| < \infty$ with $\alpha >1$, we can prove similarly the existence of a sequence $\{ v_n \}_{n=1}^\infty \subset C_{0,\sigma}^\infty (\Omega)$ such that $\nabla v_n \rightarrow \nabla v$ in $L^2 (\Omega)$ and $(|x|+1)^\alpha v_n \rightarrow (|x|+1)^\alpha v$ weakly $*$ in $L^\infty (\Omega)$ as $n \rightarrow \infty$.
\end{remark}

\begin{remark}
This proposition is also valid even if $\Omega$ is an exterior domain in $\mathbb{R}^n$ with $n \ge 3$. Indeed, we can easily verify that Lemmas \ref{plane}, \ref{star-shaped}, \ref{localization} and \ref{bounded} are valid even in $\mathbb{R}^n$ and $D \subset \mathbb{R}^n$, and the proof of Proposition \ref{exterior} still holds for $\Omega \subset \mathbb{R}^n$.
\end{remark}

\section{Proof of Theorem \ref{theorem}}
In this section we give the proof of our main result. If $u,v \in \dot{H}_{0,\sigma}^{1,S} (\Omega)$ are symmetric weak solutions of \eqref{eq:NS}, then $u$ and $v$ satisfy
\begin{equation}
\label{eq:u weak}
  (\nabla u, \nabla \varphi) + (u \cdot \nabla u, \varphi) = (f,\varphi) \quad \text{for all } \varphi \in C_{0,\sigma}^\infty (\Omega)
\end{equation}
and
\begin{equation}
\label{eq:v weak}
  (\nabla v, \nabla \widetilde{\varphi}) + (v \cdot \nabla v, \widetilde{\varphi}) = (f,\widetilde{\varphi}) \quad \text{for all } \widetilde{\varphi} \in C_{0,\sigma}^\infty (\Omega)
\end{equation}
respectively. We take $u$ and $v$ as test functions, respectively, in \eqref{eq:v weak} and \eqref{eq:u weak}. Notice that we have almost no information on the class of the nonlinear term $u \cdot \nabla u$. The assumption $\sup_{x \in \Omega} (|x|+1)|v(x)| < \infty$ and Proposition \ref{exterior} play an important role to overcome this difficulty and we also need the Hardy inequality for symmetric functions, which is due to Galdi \cite[Lemma 3.1]{G2}.

\begin{lemma}[\cite{G2}]
\label{Hardy}
Let $\Omega$ be a symmetric exterior domain with locally Lipschitz boundary and assume that $u \in \dot{H}_0^1 (\Omega)$ satisfies the symmetry property \eqref{symmetry of function3}. Then there exists a constant $C=C(\Omega)$ such that
\begin{equation*}
  \int_{\Omega} \frac{|u(x)|^2}{|x|^2} \,dx \le C \| \nabla u \|_2^2 .
\end{equation*}
\end{lemma}

\begin{remark}
If $\Omega$ and $u$ are not symmetric, then there holds
\begin{equation}
\label{remark3}
  \int_{\Omega} \frac{|u(x)|^2}{|x|^2 (1+ | \log |x| | )^2} \,dx \le C \| \nabla u \|_2^2.
\end{equation}
\end{remark}

With the aid of this lemma, we can take $u$ and $v$ as test functions. We also prove that the weak solution $v$ satisfies the energy equality. 

\begin{lemma}
\label{weak}
Let $\Omega$ be a symmetric exterior domain with Lipschitz boundary. Suppose $u,v \in \dot{H}_{0,\sigma}^{1,S} (\Omega)$ are symmetric weak solutions of \eqref{eq:NS} with $\sup_{x \in \Omega} (|x|+1)|v(x)| < \infty$. Then we have
\begin{equation}
\label{eq:u weak2}
  (\nabla u, \nabla v) + (u \cdot \nabla u,v) = (f,v),
\end{equation}
\begin{equation}
\label{eq:v weak2}
  (\nabla v, \nabla u) - (v \cdot \nabla u,v) = (f,u).
\end{equation}
In addition, $v$ satisfies the energy equality
\begin{equation}
 \label{eq:EE}
 \| \nabla v \|_2^2 =(f,v).
\end{equation}
\end{lemma}

\begin{proof}
According to Proposition \ref{exterior}, there exists a sequence $\{ v_n \}_{n=1}^\infty \subset C_{0,\sigma}^\infty (\Omega)$ such that $\nabla v_n \rightarrow \nabla v$ in $L^2 (\Omega)$ and $(|x|+1)v_n \rightarrow (|x|+1)v$ weakly $*$ in $L^\infty (\Omega)$ as $n \rightarrow \infty$. We substitute $v_n$ for $\varphi$ in \eqref{eq:u weak} to obtain
\begin{equation}
\label{eq:u weak3}
  (\nabla u, \nabla v_n) + (u \cdot \nabla u,v_n) = (f,v_n),
\end{equation}
and we write
\begin{equation*}
 (u \cdot \nabla u,v_n) = \left(\frac{u}{|x|+1} \cdot \nabla u,(|x|+1)v_n \right) .
\end{equation*}
By Lemma \ref{Hardy} we see that
\begin{equation*}
 \left\| \frac{u}{|x|+1} \cdot \nabla u \right\|_1 \le \left\| \frac{u}{|x|+1} \right\|_2 \| \nabla u \|_2 \le C \| \nabla u \|_2^2,
\end{equation*}
which together with the property of $v_n$ yields
\begin{equation*}
  (u \cdot \nabla u,v_n) \rightarrow (u \cdot \nabla u,v) \quad \text{as } \ n \rightarrow \infty.
\end{equation*}
Hence we derive \eqref{eq:u weak2} by letting $n \rightarrow \infty$ in \eqref{eq:u weak3}. On the other hand, $v \in L^4 (\Omega)$ in particular and by the class of $u$ we can take a sequence $\{ u_n \}_{n=1}^\infty \subset C_{0,\sigma}^\infty (\Omega)$ such that $\nabla u_n \rightarrow \nabla u$ in $L^2 (\Omega)$ as $n \rightarrow \infty$. We  insert $u_n$ into $\widetilde{\varphi}$ in \eqref{eq:v weak} and integrate the second term by parts to get
\begin{equation*}
  (\nabla v, \nabla u_n) - (v \cdot \nabla u_n ,v) = (f,u_n).
\end{equation*}
Since
\begin{equation*}
 |(v \cdot \nabla u_n, v)| \le \| v \|_4^2 \| \nabla u_n \|_2 ,
\end{equation*}
we obtain \eqref{eq:v weak2} by passing to the limit $n \rightarrow \infty$.

Next, we show the energy equality. Since $v \in \dot{H}_{0,\sigma}^1 (\Omega) \cap L^4 (\Omega)$ and $C_{0,\sigma}^\infty (\Omega)$ is dense in $\dot{H}_{0,\sigma}^1 (\Omega) \cap L^4 (\Omega)$ (\cite[Theorem 2]{KS}), there exists a sequence $\{ \tilde{v}_n \}_{n=1}^\infty \subset C_{0,\sigma}^\infty (\Omega)$ such that $\nabla \tilde{v}_n \rightarrow \nabla v$ in $L^2 (\Omega)$ and $\tilde{v}_n \rightarrow v$ in $L^4 (\Omega)$ as $n \rightarrow \infty$. An integration by parts gives
\begin{equation*}
 (v \cdot \nabla v,\tilde{v}_n) = - (v \cdot \nabla \tilde{v}_n ,v).
\end{equation*}
By the estimates
\begin{equation*}
 |(v \cdot \nabla v,\tilde{v}_n)| \le \| v \|_4 \| \nabla v \|_2 \| \tilde{v}_n \|_4 , \quad |(v \cdot \nabla \tilde{v}_n,v)| \le \| v \|_4^2 \| \nabla \tilde{v}_n \|_2 ,
\end{equation*}
we deduce
\begin{equation*}
 (v \cdot \nabla v,\tilde{v}_n) \rightarrow (v \cdot \nabla v,v), \quad -(v \cdot \nabla \tilde{v}_n,v) \rightarrow -(v \cdot \nabla v,v)
\end{equation*}
as $n \rightarrow \infty$. Therefore
\begin{equation*}
 (v \cdot \nabla v,v) =0.
\end{equation*}
Taking $\tilde{v}_n$ as a test function in \eqref{eq:v weak} and then letting $n \rightarrow \infty$, we derive the energy equality \eqref{eq:EE}.
\end{proof}

\begin{remark}
As we can see in the proof, we can prove this lemma without the symmetry of $v$. We need the symmetry property of $v$ to apply Lemma \ref{Hardy} in the proof of Theorem \ref{theorem} below.
\end{remark}

\begin{remark}
\label{remark3}
If $\sup_{x \in \Omega} (|x|+1)^\alpha |v(x)| < \infty$ with $\alpha >1$, we can prove \eqref{eq:u weak2}, \eqref{eq:v weak2} and \eqref{eq:EE} without assuming any symmetry. With the aid of Remark \ref{alpha}, we use the inequality \eqref{remark3}, instead of Lemma \ref{Hardy}, to take $v$ as a test function in \eqref{eq:u weak}. The similar argument to the proof of Theorem \ref{theorem} below yields Remark \ref{remark2}.
\end{remark}

Following the argument due to Miyakawa \cite{Mi}, we give the proof of Theorem \ref{theorem}.

\begin{proof}[Proof of Theorem \ref{theorem}]
Put $w:=u-v$. We first show that
\begin{equation}
\label{w}
 (w \cdot \nabla v,v)=0.
\end{equation}
We apply Proposition \ref{exterior} to take a sequence $\{ v_n \}_{n=1}^\infty \subset C_{0,\sigma}^\infty (\Omega)$ such that $\nabla v_n \rightarrow \nabla v$ in $L^2 (\Omega)$ and $(|x|+1)v_n \rightarrow (|x|+1)v$ weakly $*$ in $L^\infty (\Omega)$ as $n \rightarrow \infty$. By an integration by parts, we have
\begin{equation}
\label{weak of w}
 (w \cdot \nabla v,v_n) = - (w \cdot \nabla v_n ,v).
\end{equation}
Since $w$ satisfies the symmetry property \eqref{symmetry of function3}, the same calculation as the proof of \eqref{eq:u weak2} yields $(w \cdot \nabla v,v_n) \rightarrow (w \cdot \nabla v,v)$ as $n \rightarrow \infty$. On the other hand, by Lemma \ref{Hardy} we see
\begin{align*}
  |(w \cdot \nabla v_n ,v)| &= \left| \left(\frac{w}{|x|+1} \cdot \nabla v_n, (|x|+1)v \right) \right| \\
 &\le \sup_{x \in \Omega} (|x|+1)|v(x)| \left\| \frac{w}{|x|+1} \cdot \nabla v_n \right\|_1 \\
 &\le C \sup_{x \in \Omega} (|x|+1)|v(x)| \| \nabla w \|_2 \| \nabla v_n \|_2 ,
\end{align*}
which implies $-(w \cdot \nabla v_n, v) \rightarrow -(w \cdot \nabla v,v)$ as $n \rightarrow \infty$. Hence passing to the limit $n \rightarrow \infty$ in \eqref{weak of w}, we obtain \eqref{w}.

According to Lemma \ref{weak}, we have
\begin{equation}
\label{eq:u weak 6}
  (\nabla u, \nabla v) = -(u \cdot \nabla u,v) + (f,v)
\end{equation}
and
\begin{equation}
\label{eq:v weak 6}
  (\nabla v, \nabla u)  = (v \cdot \nabla u,v)+ (f,u).
\end{equation}
It follows from \eqref{w}, \eqref{eq:u weak 6} and \eqref{eq:v weak 6} that
\begin{equation*}
  2(\nabla u, \nabla v) = - (w \cdot \nabla w,v) + (f,u) + (f,v).
\end{equation*}
Thus the energy inequality $\| \nabla u \|_2^2 \le (f,u)$, the energy equality \eqref{eq:EE} and Lemma \ref{Hardy} lead us to
\begin{align*}
  \| \nabla w \|_2^2 &= \| \nabla u \|_2^2 + \| \nabla v \|_2^2 - 2(\nabla u, \nabla v) \\
                              &\le (w \cdot \nabla w,v) \\
                              &\le \sup_{x \in \Omega} (|x|+1)|v(x)| \left\| \frac{w}{|x|+1} \cdot \nabla w \right\|_1  \\
                              &\le C \delta \| \nabla w \|_2^2 ,
\end{align*}
where $C=C(\Omega)$ is the constant in Lemma \ref{Hardy}. Now we take the constant $\delta$ so that
\begin{equation*}
 0< \delta < \frac{1}{C}.
\end{equation*}
Then we derive
\begin{equation*}
 \| \nabla w \|_2 = 0.
\end{equation*}
Consequently, $w$ is a constant in $\Omega$, and by the boundary condition we conclude $w=0$ in $\Omega$. This completes the proof of Theorem \ref{theorem}.
\end{proof}
\subsection*{Acknowledgment}
This work was partly supported by Grant-in-Aid for JSPS Fellows Number 25002702. The author would like to thank Professor R. Farwig and Professor M. Yamazaki for useful comments.

\end{document}